\documentclass[12pt, a4paper]{article}
\usepackage[T1]{fontenc}
\usepackage[french]{babel}
\usepackage{amsmath,amsthm}
\usepackage{amssymb}
\usepackage{amscd}
\usepackage{euscript}
\usepackage[latin1]{inputenc}
\usepackage{amsfonts}
\usepackage[matrix,arrow,curve]{xy}

\newcommand{\C}{\mathbb{C}}
\newcommand{\Q}{\mathbb{Q}}
\newcommand{\Z}{\mathbb{Z}}

\newcommand{\ra}{\rightarrow}

\theoremstyle{plain}
\newtheorem{theo}{Th\'eor\`eme}[section]
\newtheorem{lem}[theo]{Lemme}
\newtheorem{prop}[theo]{Proposition}
\newtheorem{cor}[theo]{Corollaire}

\theoremstyle{definition}
\newtheorem{df}[theo]{D\'efinition}

\newtheorem{rem}[theo]{Remarque}

\title{La conjecture de Tate enti\`ere pour les cubiques de dimension quatre}

\author{Fran\c{c}ois Charles\thanks{
Universit\'e de Rennes1, IRMAR -- UMR 6625 du CNRS, Campus de Beaulieu, 35042 Rennes cedex, France,
francois.charles@univ-rennes1.fr} \and Alena Pirutka \thanks{
Université de Strasbourg, IRMA -- UMR 7501 du CNRS, 7 rue René Descartes, 67084 Strasbourg cedex, France,
pirutka@math.unistra.fr}}

\begin{document}
\maketitle

\begin{abstract}
Dans ce texte, on établit une version entière de la conjecture de Tate pour les cycles de codimension $2$ sur  une  hypersurface cubique lisse $X$ de $\mathbb P^5$ sur une clôture algébrique d'un corps de type fini sur son sous-corps premier et de caractéristique différente de $2$ et $3$. La preuve s'appuie sur la conjecture de Tate à coefficients rationnels prouvée dans ce cas par le premier auteur et sur un argument de géométrie complexe dû à Voisin.

\bigskip

We prove the integral Tate conjecture for cycles of codimension $2$ on smooth cubic fourfolds over an algebraic closure of a field finitely generated over its prime subfield and of characteristic different from $2$ and $3$. The proof relies on the Tate conjecture with rational coefficients proved in that setting by the first author and on an argument of Voisin coming from complex geometry.

\end{abstract}

\section{Introduction}

Soient $k$ un corps de type fini sur son sous-corps premier, $\overline k$ une clôture algébrique de $k$, $\ell$ un nombre premier différent de la caractéristique de $k$ et $X$ une variété projective lisse sur $k$. Notons $G$ le groupe de Galois absolu de $k$. Pour tout entier positif $r$, la conjecture de Tate \cite{Ta65} affirme que l'application classe de cycle à coefficients rationnels et à valeurs dans la cohomologie étale $\ell$-adique
\begin{equation}\label{classecycle}
CH^r(X)\otimes \Q_{\ell} \ra  H^{2r}(X_{\bar k}, \Q_{\ell}(r))^G,
\end{equation}
est surjective.\\

Il est bien connu que la version entière 
qui affirme la surjectivité de l'application
\begin{equation}\label{classecycleZraff}
CH^r(X)\otimes \Z_{\ell} \ra H^{2r}(X_{\overline k}, \Z_{\ell}(r))^G
\end{equation}
n'est en général pas vérifiée. Ce problème est discuté en détail dans \cite{CTSz10}. Si $r$ vaut $1$, un argument classique montre que la surjectivité des applications $(\ref{classecycle})$ et $(\ref{classecycleZraff})$ est équivalente.

Dans le problème entier énoncé au paragraphe précédent on peut distinguer deux questions. La première est l'analogue de la conjecture de Hodge entière, et demande l'existence de cycles sur la clôture algébrique du corps de base $k$. La seconde demande que ces cycles descendent à $k$ lui-même. De façon plus précise, la première de ces questions demande si l'application
\begin{equation}\label{classecycleZ}
CH^r(X_{\bar k})\otimes \Z_{\ell} \ra \bigcup_U H^{2r}(X_{\overline k}, \Z_{\ell}(r))^U
\end{equation}
est surjective. Pour les classes de courbes -- soit $r=\mathrm{dim}(X)-1$ -- Schoen a montré dans \cite{Sch98} que l'application (\ref{classecycleZ}) est surjective si la conjecture de Tate est vraie pour les diviseurs.

Le cas de la codimension au moins $2$ est plus difficile. D'une part, certains contre-exemples à la conjecture de Hodge entière, comme celui d'Atiyah et Hirzebruch \cite{AH62} s'adaptent à la caractéristique positive (voir \cite{CTSz10}) pour fournir un contre-exemple à la conjecture de Tate entière. D'autre part, et c'est un problème qui ne se pose pas directement sous cette forme pour la conjecture de Hodge, il est faux en général que l'application $CH^r(X)\otimes \Z_\ell\ra CH^r(X_{\bar k})^G\otimes \Z_{\ell}$ soit surjective si $r$ est différent de $0, 1$ et de la dimension de $X$, comme montré dans \cite{Pi11}.

\bigskip

La conjecture de Hodge entière pour les variétés dont la dimension de Kodaira est petite a été étudiée par Voisin et Höring-Voisin dans une série d'articles \cite{HV11, Vo06, Vo07, Vo13} qui aboutissent à la preuve de la conjecture de Hodge pour les classes entières dans la cohomologie en degré $4$ pour les variétés de Calabi-Yau et les variétés unireglées de dimension $3$, ainsi que pour certaines variétés rationnellement connexes et certaines fibrations en cubiques de dimension $3$. Les arguments de Voisin sont largement de nature transcendante, et font appel de manière cruciale à la notion de variation de structures de Hodge. La question de la validité de la conjecture de Tate entière est néanmoins abordée dans \cite{Vo12} pour les variétés rationnellement connexes. Parimala et Suresh  \cite{PS10} établissent une version forte de la conjecture de Tate entière  pour les cycles de codimension $2$ pour $X$ une fibrations en coniques au-dessus d'une surface géométriquement réglée : pour une telle fibration 
$X$,  l'application $CH^2(X)\otimes \Z_{\ell} \ra H^{4}(X, \Z_{\ell}(2))$ est surjective. Ceci permet ausssi d'obtenir la même conclusion  pour un solide cubique lisse. Hors de ces travaux et du théorème de Schoen cité ci-dessus, il ne semble pas qu'il existe de résultat positif sur la conjecture de Tate à coefficients entiers sur les corps finis.

\bigskip

Dans \cite{Vo07}, Voisin montre que la conjecture de Hodge entière est vraie pour les cycles de codimension $2$ dans une cubique lisse de $\mathbb P^5_{\C}$. Une des clés de la preuve est la méthode des fonctions normales de Zucker \cite{Zu76} qui généralise à la codimension supérieure la démonstration de Poincaré et Lefschetz de la conjecture de Hodge pour les classes de diviseurs. Le théorème principal de cet article étend ce résultat au cas des corps finis.

\begin{theo}\label{tate-entiere}
Soit $k$ un corps de type fini sur son sous-corps premier et de caractéristique différente de $2$ et $3$. Soit $\overline k$ une clôture algébrique de $k$. Soit $X$ une  hypersurface cubique lisse de $\mathbb P^5_k$. La conjecture de Tate entière est vraie pour les cycles de codimension $2$ sur $X$. Autrement dit, pour tout nombre premier $\ell$ différent de la caractéristique de $k$, l'application classe de cycle
\begin{equation}\label{cycle}
CH^2(X_{\bar k})\otimes\Z_{\ell}\to \bigcup_U H^{4}(X_{\bar k}, \Z_{\ell}(2))^U
\end{equation}
est surjective, où $G$ est le groupe de Galois absolu de $k$  et $U$ parcourt le système des sous-groupes ouverts de $G$.
\end{theo}

Il est à remarquer que le cas où la caractéristique de $k$ est nulle est déjà connu : il résulte de la conjecture de Tate à coefficients rationnels prouvé dans ce cas par André dans \cite[1.6.1]{An96} -- pour les corps de nombres, et le cas général s'en déduit -- et de la conjecture de Hodge entière pour les cubiques de dimension $4$ prouvée par Voisin dans \cite{Vo07}.

La version rationnelle de la conjecture de Tate pour  les cubiques de dimension $4$ sur les corps finis de caractéristique au moins $5$ est démontrée dans \cite[Corollaire 6]{Ch12} -- le cas des corps de type fini sur leur sous-corps premier s'en déduit par des techniques standard -- et dans \cite[Théorème 5.14]{Pera13}. Il s'agit du point de départ de notre résultat.

Sa version rationnelle étant acquise, la preuve du théorème s'inspire de celle de \cite{Vo07, Vo13}. Il s'agit d'abord d'utiliser la méthode de Zucker en associant à une classe Galois-invariante une section de la fibration en jacobiennes intermédiaires associée à un pinceau de Lefschetz de $X$. Ensuite, un résultat de Markushevich et Tikhomirov \cite{MT01} permet de construire une famille de cycles algébriques à partir de toute section comme ci-dessus.

En général, les notions de jacobienne intermédiaire et de fonction normale, qui sont des objets de géométrie complexe, n'ont pas d'analogue en caractéristique positive. Dans notre cas, c'est la description due à Clemens et Griffiths \cite{CG72}, de la jacobienne intermédiaire d'une cubique de dimension $3$ qui permet de donner un sens à la stratégie ci-dessus. \`A notre connaissance, il s'agit de la première utilisation des fonctions normales pour la construction de cycles algébriques sur un corps fini.

\bigskip

\paragraph{Remerciements.} Ce travail a largement bénéficié de deux séjours du premier auteur à l'Université de Strasbourg, que nous remercions pour son hospitalité. Nous remercions les ANR CLASS et Positive pour leur soutien financier. Nous remercions Jean-Louis Colliot-Thélène, Bruno Kahn, Tam\'as Szamuely et Claire Voisin pour leur lecture détaillée et leurs remarques précises qui nous ont permis d'améliorer et corriger la première version de ce texte.

\section{Fonctions normales pour une fibration en cubiques}

Dans cette section, nous considérons une situation un peu plus générale que celle du théorème \ref{tate-entiere}. Soit $k$ un corps de type fini sur son sous-corps premier, ou un corps algébriquement clos. Soit $\pi : Y\ra T$ un morphisme entre les variétés projectives lisses sur $k$. On suppose que la fibre générique de $\pi$ est une hypersurface cubique lisse de $\mathbb P^4$. Autrement dit, $\pi$ est une fibration en solides cubiques au-dessus de $T$. Soit $U$ l'ouvert de $T$ au-dessus duquel $\pi$ est un morphisme lisse.

Pour simplifier les notations dans ce qui suit, et comme cette condition sera satisfaite dans le cas du théorème \ref{tate-entiere}, nous supposerons en outre que le groupe $H^3(Y_{\overline k}, \Z_{\ell})$ est nul, où $\ell$ est un nombre premier différent de la caractéristique de $k$ et où $\overline k$ est une clôture algébrique de $k$. Cette hypothèse garantit que le système local $R^3\pi_*\Z_{\ell}$ sur $U$ n'a pas de partie constante. En général, les énoncés qui suivent valent en quotientant par cette partie constante.

\bigskip

Supposons d'abord que le corps $k$ soit le corps des nombres complexes $\C$. Nous répétons ici les constructions décrites dans \cite{Zu76, Vo13}. Les chapitres 12 et 19 de \cite{Vo02} contiennent également des précisions sur ce qui suit.

Si $Y_t$ est une fibre lisse de $\pi$ au-dessus d'un point complexe $t$ de $T$, on dispose, suivant Griffiths \cite{Gr68}, de la jacobienne intermédiaire
$$J(Y_t)=\frac{H^3(Y_t, \C)}{F^2H^3(X, \C)\oplus H^3(X, \Z)},$$
où $F^{\bullet}$ est la filtration de Hodge. Il s'agit d'un tore complexe. D'après \cite{Gr68} encore, ce qui précède vaut en famille, et l'on obtient une fibration $J\ra U$ en jacobiennes intermédiaires.

Soit maintenant $\alpha$ une classe de Hodge dans $H^4(Y, \Z(2))$ telle que la restriction de $\alpha$ à $H^4(Y_t, \Z(2))$ soit nulle pour une -- ou toute -- fibre lisse $Y_t$ de $\pi$. La classe $\alpha$ induit dans ce cas une fonction normale $\nu_{\alpha}$ associée à $J\ra U$, c'est-à-dire une section de $J\ra U$ vérifiant la propriété d'horizontalité.

Dans le cas où $\alpha$ est la classe de cohomologie d'un cycle $Z$ de codimension $2$ dans $Y$, que l'on supposera pour simplifier plat au-dessus de $T$, la fonction normale $\nu_{\alpha}$ s'obtient comme suit. Si $t$ est un point de $U$, soit $Z_t\in CH^2(Y_t)$ la restriction de $Z$ à la fibre $Y_t$. Le cycle $Z_t$ est par hypothèse homologue à zéro dans $Y_t$. Son image par l'application d'Abel-Jacobi de Griffiths
$$\Phi_t : CH^2(Y_t)_{hom}\ra J(Y_t)$$
est bien définie. C'est la valeur en $t$ de la fonction normale $\nu_{\alpha}$.

Ce qui précède s'applique sans hypothèse sur la fibre générale de $Y\ra T$. Dans cette généralité, il s'agit cependant d'une construction transcendante qui n'a pas de pendant connu sur un corps plus général que celui des complexes. Dans notre cas, la situation est en fait algébrique. En effet, si $t$ est un point de $U$, il résulte de l'annulation du groupe $H^3(Y_t, \mathcal O_{Y_t})$ que la jacobienne intermédiaire $J(Y_t)$ est une variété abélienne par \cite[12.2.2]{Vo02}. Plus généralement, la fibration en jacobiennes intermédiaires $J\ra U$ est un schéma abélien au-dessus de $U$. La suite de cette section est consacrée à la description de la fonction normale associée à un cycle comme ci-dessus au-dessus d'un corps  $k$  de type fini sur son sous-corps premier.

Nous utiliserons dans ce qui suit la cohomologie continue au sens de Jannsen \cite{Ja88}. Si $k$ est fini, il s'agit simplement de la cohomologie étale $\ell$-adique. Soit $\ell$ un nombre premier différent de la caractéristique de $k$. Dans ce contexte, l'analogue de l'application d'Abel-Jacobi de Griffiths est l'application d'Abel-Jacobi $\ell$-adique définie d'abord par Bloch \cite{Bl84}. On renvoie à \cite{Ch10} pour une discussion de la notion de fonction normale dans ce contexte. Soit $\overline k$ une clôture algébrique de $k$, et soit $\ell$ un nombre premier différent de la caractéristique de $k$.

Soit maintenant $\alpha$ un élément du groupe de cohomologie $\ell$-adique $H^4(Y, \Z_{\ell}(2))$ dont la restriction à une -- ou toute -- fibre géométrique lisse de $\pi$ est nulle. Via la suite spectrale de Leray pour le morphisme $\pi$
$$E_2^{pq}=H^p(U, R^q\pi_*\mathbb \Z_{\ell}(2))\Rightarrow  H^{p+q}(Y_U, \mathbb \Z_{\ell}(2)),$$ la classe $\alpha$ induit un élément du groupe $H^1(U, R^3\pi_*\Z_{\ell}(2))$, qui ne dépend que de l'image de $\alpha$ dans $H^4(Y_{\overline k}, \Z_{\ell}(2))$ puisque le groupe $H^3(Y_{\overline k}, \Z_{\ell}(2))$ est nul par hypothèse.

Au sens de \cite{Ch10} par exemple, il s'agit de la classe de cohomologie de la fonction normale qu'il nous reste encore à construire. Pour ce faire, il faut décrire de manière géométrique la jacobienne intermédiaire d'une cubique de dimension $3$. Cette description est essentiellement due à Clemens et Griffiths \cite{CG72}.

\bigskip

Rappelons que d'après \cite[1.12]{AK77}, le schéma paramétrant les droites contenues dans une hypersurface cubique lisse de dimension $3$ est une surface lisse -- c'est la {\it surface de Fano} de la cubique. Soit $\psi : F\ra U$ la surface de Fano relative de $\pi$. Le morphisme $\psi$ est projectif et lisse, de dimension relative $2$. Ses fibres sont les surfaces de Fano des fibres de $\pi$. Soit $V$ la variété d'incidence associée au-dessus de $U$, et $p, q$ les deux morphismes canoniques de $V$ dans $F$ et $Y_U$ respectivement. La dimension relative de $V$ est $3$.  Cette situation correspond au diagramme suivant :

\footnotesize
\begin{equation}\label{inc}\xymatrix{
& V \ar[dl]_{p}\ar[dr]^{q} & & \\
 F\ar[dr]_{\psi}& & Y_U\ar[dl]^{\pi} &  \\
& U & &
}
\end{equation}
\normalsize
D'après \cite[Théorème 11.19]{CG72}, le morphisme
$$p_*q^* : R^3\pi_*\Z_{\ell}(2)\ra R^1\psi_*\Z_{\ell}(1)$$
est un isomorphisme de faisceaux étales sur $U$. Pour le voir, il suffit de montrer que c'est un isomorphisme en chaque fibre géométrique de $\pi$. Comme les hypersurfaces cubiques se relèvent en caractéristique zéro, on se ramène au cas des cubiques lisses de dimension $3$ sur $\C$, où notre énoncé s'obtient à partir de celui prouvé par Clemens et Griffiths par dualité de Poincaré.

\bigskip

Soit maintenant $J\ra U$ le schéma  ${\bf Pic}^{\tau}(F/U)$. Pour tout entier positif $r$, la suite exacte de Kummer
\begin{equation}\label{kumJ}
 0\to R^1\psi_*\Z/\ell^r\Z(1)\to J\stackrel{\ell^r}{\to}J\to 0
 \end{equation}
induit une application
$$H^0(U,J)/\ell^r\to H^1(U, R^1\psi_*\Z/\ell^r\Z(1)).$$
En passant à la limite, on obtient une application
\begin{equation}\label{kum}
H^0(U,J)\otimes \mathbb Z_\ell\to H^1(U,R^1\psi_*\mathbb Z_{\ell}(1)).
\end{equation}
Si $\nu$ est un élément de $H^0(U,J)\otimes \mathbb Z_\ell$, on note $[\nu]\in H^1(U, R^3\pi_*\Z_{\ell}(2))$ son image par le composé du morphisme précédent avec l'isomorphisme
$$(p_*q^*)^{-1} :  H^1(U,R^1\psi_*\mathbb Z_{\ell}(1)) \ra H^1(U, R^3\pi_*\Z_{\ell}(2)).$$

\begin{df}
Soit $\alpha$ un élément du groupe $H^4(Y, \Z_{\ell}(2))$ dont la restriction aux fibres géométriques lisses de $\pi$ est nulle. On dit qu'un élément $\nu$ de $H^0(U,J)\otimes \mathbb Z_\ell$ est une \textit{fonction normale} associée à $\alpha$ si $[\nu]$ est l'image de $\alpha$ dans $H^1(U, R^3\pi_*\Z_{\ell}(2))$.
\end{df}

Contrairement à la situation sur le corps des complexes, on ne sait pas construire de fonction normale associée à une classe de cohomologie. Le résultat suivant montre que c'est possible pour les classes de cycles algébriques.

\begin{prop}\label{nf}
Soit $\alpha$ un élément du groupe $H^4(Y, \Z_{\ell}(2))$ dont la restriction aux fibres géométriques lisses de $\pi$ est nulle. Si $\alpha$ est la classe de cohomologie d'un cycle de codimension $2$ sur $Y$, alors il existe une fonction normale $\nu_{\alpha}$ associée à $\alpha$. De plus, $\nu_{\alpha}$ appartient à l'image de $H^0(U, J)$ dans $H^0(U,J)\otimes \mathbb Z_\ell$.
\end{prop}

\begin{proof}
Donnons une construction explicite de $\nu_{\alpha}\in H^0(U, J)$. Soit $Z\in CH^2(Y)$ un cycle algébrique de classe de cohomologie $\alpha$. On note aussi $Z$ la restriction de $Z$ à l'ouvert $Y_U$ par abus de notation.

Le cycle $Z'=p_*q^*Z$ est un élément de $CH^1(F)$, soit une classe d'équivalence rationnelle de diviseurs sur $F$. Comme la restriction de $Z$ aux fibres géométriques lisses de $\pi$ est homologiquement équivalente à zéro, il en va de même de la restriction de $Z'$ aux fibres géométriques de $\psi$. Autrement dit, le diviseur $Z'$ est de degré zéro sur les fibres géométriques de $\psi$, et correspond donc à une section $\nu_{\alpha}$ de $J={\bf Pic}^{\tau}(F/U).$

\bigskip

Montrons que l'image de $\nu_{\alpha}$ dans $H^0(U,J)\otimes \mathbb Z_\ell$ est la fonction normale associée à $\alpha$. Soit $\beta$ la classe de cohomologie du cycle $Z'$ dans $F$. Par la suite spectrale de Leray associée à $\psi$, la classe $\beta$ induit un élément de $H^1(U,R^1\psi_*\mathbb Z_{\ell}(1))$. Par fonctorialité de la correspondance $p_*q^*$, il s'agit de montrer que l'image de $\nu_{\alpha}$ par  la flèche (\ref{kum}) est égale à l'image de $\beta$ dans $H^1(U,R^1\psi_*\mathbb Z_{\ell}(1))$. Pour ce faire,  il suffit de montrer l'égalité des classes de cohomologie ci-dessus au-dessus du point générique de $U$, ce qui est essentiellement d\'emontr\'e dans \cite[Appendix]{Ra95}. Un résultat semblable dans un contexte plus général est aussi discuté dans \cite[Remarque 2 après le théorème 12]{Ch10}.

\end{proof}

Enfin, on obtient ce qui suit.

\begin{prop}\label{torsion}
Soit $\alpha$ un élément du groupe $H^4(Y, \Z_{\ell}(2))$ dont la restriction aux fibres géométriques lisses de $\pi$ est nulle, et soit $N$ un élément de $\Z_{\ell}$. S'il existe une fonction normale associée à $N\alpha$, alors il existe une fonction normale associée à $\alpha$.
\end{prop}

\begin{proof}
La suite exacte longue de cohomologie associée à (\ref{kumJ}) donne la suite exacte
$$0\to \varprojlim H^0(U,J)/\ell^r\to H^1(U, R^1\psi_*\Z_{\ell}(1))\to T_{\ell}H^1(U,J)$$
où $T_{\ell}H^1(U,J)$ est le module de Tate de $H^1(U, J)$.
Comme le groupe $H^0(U,J)/\ell$ est fini d'après les théorèmes de Mordell-Weil et de Lang-Néron, l'application
$$H^0(U,J)\otimes \mathbb Z_\ell\to \varprojlim H^0(U,J)/\ell^r$$
est surjective (voir \cite[Lemme 2.1]{CTK12}).
Comme le module de Tate $T_{\ell}H^1(U,J)$  est sans torsion, le conoyau de l'application (\ref{kum}) est sans torsion.  Si l'image de $N\alpha$ dans $H^1(U, R^1\psi_*\Z_{\ell}(1))$ est dans l'image de (\ref{kum}), il en va donc de même de l'image de $\alpha$.
\end{proof}

\section{Preuve du th\'eorème}

Dans cette section on établit le théorème suivant.

\begin{theo}\label{tate-entiere1}
Soit $k$ un corps de type fini sur son sous-corps premier et de caractéristique différente de $2$ et $3$. Soit $\overline k$ une clôture algébrique de $k$. Soit $X$ une hypersurface cubique lisse de $\mathbb P^5_k$.

La conjecture de Tate entière est vraie pour les cycles de codimension $2$ sur $X$. Autrement dit, pour tout nombre premier $\ell$ différent de la caractéristique de $k$, l'application classe de cycle
$$CH^2(X_{\bar k})\otimes\Z_{\ell}\to \bigcup_U H^4(X_{\bar k},\mathbb \Z_{\ell}(2))^U$$
est surjective, où $G$ est le groupe de Galois absolu de $k$, et $U$ parcourt le système des sous-groupes ouvert de $G$.
\end{theo}

Pour prouver le théorème \ref{tate-entiere1}, nous allons appliquer à un pinceau de Lefschetz sur $X_{\bar k}$  la construction de la fonction normale de la section précédente. 

\bigskip

\begin{lem}\label{existence}
Il existe un pinceau de Lefschetz de sections hyperplanes de $X_{\bar k}$.
\end{lem}

\begin{proof}
Cet énoncé est une conséquence des théorèmes d'existence de \cite[Exposé XVII]{SGA7}. Soit $X_{\bar k}^{\vee}$ la variété duale de $X_{\bar k}$ dans l'espace projectif $\mathbb P$ des sections hyperplanes de $X_{\bar k}$. L'application de Gauss de $X_{\bar k}$ sur sa variété duale $X_{\bar k}^{\vee}$, qui à un point de $X_{\bar k}$ associe la section hyperplane découpée par l'hyperplan tangent, est finie de degré $3\cdot 2^5$ comme l'exemple 3.4 de \cite[Exposé XVII]{SGA7} le montre, elle est donc génériquement non ramifiée car la caractéristique de $k$ est différente de $2$ et $3$. D'après \cite[Exposé XVII, Proposition 3.5]{SGA7}, le lieu lisse de $X_{\bar k}^{\vee}$ est un ouvert dense $U$ de $X^{\vee}$.  Par le théorème de Bertini, on peut trouver une droite de $V_l$ qui intersecte $X_{\bar k}^{\vee}$ uniquement dans son ouvert de lissité $U$. Cette droite est un pinceau de Lefschetz vérifiant les propriétés du lemme.

\end{proof}

Conservant les notations du lemme, notons $S$ le lieu de base du pinceau de Lefschetz considéré. Soit  $\iota:Y\to X_{\bar k}$ l'\'eclatement de $X_{\bar k}$ le long de $S$. Le pinceau induit un morphisme $\pi : Y\ra \mathbb P^1_{\bar k}$ dont les fibres sont des hypersurfaces cubiques $Y_t$ de dimension $3$. Soit $U\subset \mathbb P^1_{\bar k}$ l'ouvert de lissit\'e de $\pi$ et soit $Y_U$ l'image réciproque de $U$ dans $Y$.

\lem\label{fibres}{Soit $\alpha$ un élément du groupe $H^4(X_{\bar k},\mathbb Z_{\ell}(2))$.  Il existe  une classe alg\'ebrique $\alpha_0$ dans $H^4(Y, \mathbb Z_{\ell}(2))$
 telle que pour tout point géométrique $t$ de $U$, la restriction $\alpha'_t$ de $\alpha'=\iota^*\alpha-\alpha_0$ \`a la fibre $Y_t$ est nulle.}
\proof{ Par construction, le diviseur exceptionnel de $\iota: Y\to X_{\bar k}$ est isomorphe \`a $S\times \mathbb P^1_{\bar k}$. Comme $S$ est une surface cubique sur $\bar k$, on dispose d'une droite $l$ contenue dans $S$, ce qui donne une droite, que l'on note encore $l$, dans chaque section hyperplane $Y_t$.  La restriction $\alpha_t$ de $\alpha$ à $Y_t$ s'écrit $\alpha_t=b[l]$ pour un certain entier $b$ indépendant de $t$. On peut donc prendre $\alpha_0=b[l\times \mathbb P^1]$.    \qed\\}

\bigskip

Reprenons les notations de la section précédente, notant $\psi : F \to U$ la surface de Fano relative de $\pi$, $V$ la vari\'et\'e d'incidence et $J\ra U$ le schéma  $\text{\textbf{Pic}}^{\tau}( F/U)$. On note aussi $J^0\ra U$ le schéma  $\text{\textbf{Pic}}^{0}( F/U)$.

Soit  $\alpha$ un élément de $H^4(X_{\bar k},\mathbb Z_{\ell}(2))^U$, où $U\subset G$ est un sous-groupe ouvert.  Par abus de notation, on note encore $\alpha$ la classe $\iota^*\alpha$ dans $H^4(Y,\mathbb Z_{\ell}(2))$. D'après le  lemme  \ref{fibres}, pour démontrer le théorème \ref{tate-entiere1}, on peut supposer que la restriction de $\alpha $ à toute fibre géométrique lisse de $\pi$ est nulle.
La proposition suivante permet alors d'associer à $\alpha$ une fonction normale $\nu_{\alpha}\in H^0(U,J)\otimes \mathbb Z_{\ell}$.

\prop\label{normale-cubique}{Soit $\alpha$ un élément du groupe $H^4(X_{\bar k}, \Z_{\ell}(2))$ dont la restriction aux fibres géométriques lisses de $\pi$ est nulle. Il existe alors une fonction normale  $\nu_{\alpha}\in H^0(U,J)\otimes \mathbb Z_{\ell}$  associée à $\alpha$. }
\proof{D'après la proposition \ref{torsion}, il suffit de montrer qu'il existe une fonction normale pour un multiple $N\alpha$ de la classe $\alpha$. D'après \cite[Corollary 6]{Ch12} ou \cite[5.14]{Pera13}, la conjecture de Tate \`a coefficients  rationnels vaut pour les cycles de codimension $2$ sur $X_{\bar k}$. Il existe donc un multiple de $\alpha$ qui est combinaison linéaire à coefficients entiers $\ell$-adiques de classes algébriques. L'énoncé résulte  alors des propositions \ref{nf} et \ref{torsion}.

\qed\\}

\bigskip

Pour construire un cycle algébrique sur $Y$ à partir de la fonction normale $\nu_{\alpha}$, nous adaptons un argument de Voisin dans \cite[Theorem 18]{Vo07}. Le point clé est l'existence d'un espace de modules $\mathcal M_t$ birationnel à la jacobienne intermédiaire $J_t$.

Soit $\mathcal M\to \mathbb P^1_{\bar k}$ l'espace de modules des faisceaux semi-stables de rang $2$ sans torsion  dans les fibres de $Y\to \mathbb P^1_{\bar k}$ vérifiant $c_1=0, c_2=2[l]$ dans la cohomologie des fibres géométriques de $Y\to \mathbb P^1_{\bar k}$. D'après \cite{La04, La04A}, $\mathcal M$ est un schéma projectif au-dessus de $\mathbb P^1_{\bar k}$. L'application $p_*q^*(c_2-2l)$ induit l'application d'Abel-Jacobi au-dessus de $U$
$$\Phi: \mathcal M_U\to J.$$

\prop\label{bir}{Soit $s$ un point de $U$.  Il existe une composante irréductible  $\mathcal M'_s$ de  $\mathcal M_s$ qui est birationnelle à  $J^0_s$ par l'application d'Abel-Jacobi $\Phi_s$.}
\proof{Soit $\kappa$ le corps résiduel de $s$. D'après les résultats de Markushevich, Tikhomirov et Druel \cite{MT01,Dr00}, la proposition vaut si $\kappa$ est de caractéristique nulle. En général, la cubique $Y_s$ se relève en caractéristique nulle : il existe $R$ un anneau de valuation discrète de corps résiduel $\kappa$ et de corps des fractions $K$ de caractéristique nulle et un $R$-schéma $\mathcal Y$ dont la fibre spéciale est la cubique $Y_s$ et la fibre générique est une cubique lisse sur $K$. Toujours d'après \cite{La04, La04A}, on dispose d'une famille $\mathcal M_{R}/R$ d'espaces de modules des faisceaux semi-stables sans torsion avec $c_1=0, c_2=2l$ et d'une application d'Abel-Jacobi $\Phi_{R}:\mathcal M_{R}\to \mathcal J$,  où  $\mathcal J/R$ est la famille des jacobiennes intermédiaires de  $\mathcal Y$. Puisque $K$ est de caractéristique nulle, l'application $\Phi_{K}$ est birationnelle d'après \cite{MT01,Dr00}. L'application inverse
$\mathcal J\dashrightarrow \mathcal M_{R}$ est définie en tout point de codimension $1$ de $\mathcal J^0$, puisque  $\mathcal J^0$ est normale et $\mathcal M_R$ est propre. Elle est donc définie au point générique de la fibre spéciale $J^0_s$, d'où le résultat.
\qed \\}

\cor\label{section}{Soit $\nu$ une section de $J^0$ au-dessus de $U$. Alors $\nu$ se relève en une section $\nu'$ de la flèche composée $\mathcal M_U\ra U$.

\proof{Il suffit de construire la section $\nu'$ au-dessus du point générique $\eta$ de $U$. Par la proposition précédente, il existe une composante  $\mathcal M'_{\eta}$ de  $\mathcal M_{\eta}$ telle que  l'application $\Phi_{\eta}: \mathcal M'_{\eta}\to J^{0}_{\eta}$ est un isomorphisme birationnel. L'énoncé est donc une conséquence   du lemme de Nishimura \cite{Ni55}, comme $J{^0}_{\eta}$ est une variété normale  et  $\mathcal M'_{\eta}$ est propre. \qed\\}

\begin{cor}\label{nf-existe}
Soit $\nu$ une section de $J^0$ sur $U$. Il existe une classe algébrique $\alpha$ dans $H^4(Y, \Z_\ell(2))$, de restriction nulle aux fibres géométriques lisses de $\pi$, telle que $\nu$ est une fonction normale associée à $\alpha$.
\end{cor}

\begin{proof}

D'après le corollaire \ref{section}, on peut trouver une section $\nu': U \to\mathcal M_U$ qui relève $\nu$. Si l'espace de modules $\mathcal M$ était un espace de modules fins, c'est-à-dire s'il existait une famille universelle de faisceaux sur $\mathcal M$ -- ou même sur le lieu correspondant aux points stables -- le pull-back de cette famille à $\nu'$ nous fournirait la famille de courbes cherchée.

Il faut montrer comment contourner ce problème. Nous renvoyons à \cite[Section 4]{La04} et à \cite[Chapitre 4]{HL10} pour une construction détaillée de $\mathcal M$. Le schéma $\mathcal M$ est le quotient, au sens de la théorie géométrique des invariants de \cite{Mu65}, d'un ouvert d'un schéma $Q=Quot(\mathcal O_Y(-N)^r, P)$ au-dessus de $\mathbb P^1_{\bar k}$ qui paramètre les quotients de $ \mathcal O_Y(-N)^r$ de polynôme de Hilbert $P$ par le groupe $GL(r)$ -- $r$, $N$ et $P$ étant choisis de manière convenable.

Soit $\eta$ le point générique de $U$, soit $K$ le corps $\bar k(t)$ et soit $x$ le point de $\mathcal M_{\eta}(K)$ image de $\eta$ par $\nu'$. D'après \cite[Lemme 2.1]{DN89}, qui s'applique dans notre contexte, il existe une unique orbite fermée de $GL(r)$ dans $Q_{\bar K}$ au-dessus de $x$. En particulier, cette orbite est définie sur une clôture parfaite $K^p$ de $K$. Notons $C$ cette orbite. Comme la dimension cohomologique de $K$, et donc de $K^p$, est $1$, l'ensemble des $K^p$-points de $C$ est non-vide d'après le théorème de Springer  \cite[p.134]{Se94}. Ainsi, on peut trouver une courbe propre et lisse $D$ sur $\bar k$ munie d'un morphisme $D\to \mathbb P^1_{\bar k}$ de degré $d$, o\`u $d=1$ si $k$ est de caractéristique nulle  et $d$ est une puissance de la caractéristique de $k$  sinon, et  un morphisme $D\to Q$ qui rel\`eve la section $\nu'$.

Même s'il n'existe pas nécessairement de fibré universel au-dessus de l'espace de modules $\mathcal M$, le schéma $Q$ est lui un schéma Quot, au-dessus duquel l'existence d'un fibré universel est toujours vérifiée par définition -- voir \cite{GrothendieckTDTE} et le plus récent \cite[2.2]{HL10}. Soit $\mathcal E$ le faisceau sur $Y\times_{\mathbb P^ 1} D$ correspondant et soit $Z\in CH^2(Y)\otimes \Z_{\ell}$ défini par $Z=\frac{1}{d}pr_{Y,*}(c_2(\mathcal E)-2[l\times D])$ o\`u $pr_Y$ est le morphisme de projection $Y\times_{\mathbb P^ 1} D\to Y$. D'apr\`es la construction et la proposition \ref{nf}, $\nu$ est une fonction normale associée \`a la classe de $Z$.

 \end{proof}

\begin{proof}[Preuve du théorème \ref{tate-entiere1}]
Soit $\alpha\in H^4(X_{\bar k},\Z_{\ell}(2))$. D'après le lemme \ref{fibres}, quitte à modifier $\alpha$ par une classe algébrique, on peut supposer que la restriction de $\alpha$ aux fibres géométriques lisses de $\pi$ est nulle. D'après la proposition \ref{normale-cubique}, il existe une fonction normale $\nu_{\alpha}\in H^0(U,J)\otimes \mathbb Z_{\ell}$, associée à $\alpha$. Quitte \`a multiplier  $\alpha$ par une puissance de $p=car.k$, on peut supposer que $\nu_{\alpha}\in H^0(U,J^0)\otimes \mathbb Z_{\ell}$. Le corollaire \ref{nf-existe} montre qu'il existe une classe algébrique $\beta \in H^4(Y,\Z_{\ell}(2))$, de restriction nulle aux fibres géométriques lisses de $\pi$, dont l'image dans $H^1(U, R^3\pi_*\Z_{\ell}(2))$ est égale à l'image de $\nu_{\alpha}$. Cela signifie que la classe $\alpha-\iota_*\beta$ dans $H^4(X_{\bar k},\Z_{\ell}(2))$ s'annule dans $H^4(Y_U, \Z_{\ell}(2))$.  Le théorème  est maintenant une conséquence  de la proposition ci-dessous -- qui apparaît dans \cite{Zu76} p.203 dans le cas complexe.
\end{proof}

\begin{prop}\label{nulfibres}
Soit $\alpha\in H^4(X_{\bar k}, \Z_{\ell}(2))$. Si la restriction de $\alpha$ à $H^4(Y_{U}, \Z_{\ell}(2))$ est nulle, alors $\alpha$ est nulle.
\end{prop}

\begin{proof}
Le groupe $H^4(X_{\bar k}, \Z_{\ell}(2))$ est sans torsion, il suffit donc de travailler à coefficients rationnels. D'après \cite[Exposé XVIII, Théorème 5.6.8]{SGA7}, la suite spectrale
\begin{equation}\label{leraypinceau} E_2^{p,q} = H^p(\mathbb P^1_{\bar k}, R^q\pi_*\Q_{\ell}(2))\implies H^{p+q}(Y, \Q_{\ell}(2))\end{equation}
dégénère en $E_2$. C'est aussi une conséquence du théorème de décomposition de \cite{BBD82} : le faisceau $R\pi_*\Q_{\ell}$ se décompose dans la catégorie dérivée comme somme de ses $^pR^i\pi_*\Q_{\ell}[i]$, et, puisque l'on est au-dessus d'une courbe, faisceaux pervers et systèmes locaux se correspondent.

 La filtration sur $H^4(Y, \Q_{\ell}(2))$  induite par la suite spectrale $(\ref{leraypinceau})$ a pour termes non nuls $H^0(\mathbb P^1_{\bar k}, R^4\pi_*\Q_l(2))$, $H^1(\mathbb P^1_{\bar k}, R^3\pi_*\Q_{\ell}(2))$ et $H^2(\mathbb P^1_{\bar k}, R^2\pi_*\Q_{\ell}(2))$. Soit $j : U\hookrightarrow \mathbb P^1_{\bar k}$  l'inclusion de l'ouvert de lissité de  $\pi$. D'après \cite[Exposé XVIII, Théorème 6.3]{SGA7}, le  morphisme d'adjonction
$$R^i\pi_*\Q_{\ell}\ra j_*j^*R^i\pi_*\Q_{\ell}$$
est un isomorphisme. On a donc
$$H^0(\mathbb P^1_{\bar k}, R^4\pi_*\Q_{\ell}(2))=H^0(U,  R^4\pi_*\Q_{\ell}(2))$$
et
$$H^2(\mathbb P^1_{\bar k}, R^2\pi_*\Q_{\ell}(2))=H^2(\mathbb P^1_{\bar k}, \Q_{\ell}(1))=\Q_{\ell}.$$
De plus, la flèche
$$H^1(\mathbb P^1_{\bar k}, R^3\pi_*\Q_{\ell}(2))\to H^1(U, R^3\pi_*\Q_{\ell}(2))$$
 est injective.

Notons que l'inclusion
$$H^2(\mathbb P^1_{\bar k}, R^2\pi_*\Q_{\ell}(2))\hookrightarrow H^4(Y, \Q_{\ell}(2))$$
envoie $1$ sur la classe d'une section hyperplane d'une fibre lisse de $\pi$ dans $Y$.

Soit maintenant $\alpha$ un élément de $H^4(X_{\bar k}, \Q_{\ell}(2))$ dont  la restriction aux groupes $H^0(U,  R^4\pi_*\Q_{\ell}(2))$ et $H^1(U, R^3\pi_*\Q_{\ell}(2))$ est nulle. D'après ce qui précède,  son image  $\iota^*\alpha$ dans $H^4(Y, \Q_{\ell}(2))$ est  égale à l'image  d'un multiple d'une section hyperplane d'une fibre lisse de $\pi$.

Par construction, le morphisme $\iota:Y\to X_{\bar k}$ est l'éclatement du lieu de base $S$ du  pinceau. L'orthogonal dans $H^4(Y, \Q_{\ell}(2))$ de $\iota^*H^4(X_{\bar k},  \Q_{\ell}(2))$ est donc $H^2(S,  \Q_{\ell}(1))$, engendré par $[l\times\mathbb P^1_{\bar k}]$, où $l$ est une droite dans $S$. Cela implique que la classe $\iota^*\alpha$ est nulle, ainsi que la classe $\alpha$.
\end{proof}

\rem\label{uni4}{ De manière générale, on dispose d'une application rationnelle dominante $\mathbb P^4_{\bar k}\dashrightarrow X_{\bar k}$ de degré $2$. En effet, soient $l\subset X_{\bar k}$ une droite et $p : P \ra l$ le fibré projectif $\mathbb P(T_{X_{\bar k}})|_l\ra l$. La fibre générique de $p$ est une quadrique de dimension $3$ sur  $\bar k(t)$, une telle quadrique a un point rationnel. La variété $P$ est donc rationnelle. On a une application rationnelle dominante $P\dashrightarrow X_{\bar k}$ de degré $2$,  qui a une droite tangente à $X_{\bar k}$ en un point de $l$ associe son troisième point d'intersection avec $X_{\bar k}$.  Cette construction permet de montrer que le conoyau de l'application classe de cycle dans le théorème \ref{tate-entiere1} est de torsion $2$-primaire.}

\rem{La grande majorité de la démonstration permet d'aborder la surjectivité de l'application $CH^2(X)\otimes \Z_{\ell} \ra H^{4}(X_{\overline k}, \Z_{\ell}(r))^G$  dans le cas d'un corps de base fini, quitte à prendre une extension de degré premier à $\ell$.   Le seul endroit dans lequel il est nécessaire de travailler sur $\bar k$ est  le corollaire \ref{nf-existe}. L'obstruction à démontrer la surjectivité de (\ref{classecycleZraff}) est liée à la question de l'existence d'un cycle universel sur la jacobienne intermédiaire d'une cubique de dimension $3$ posée dans \cite[Question 1.6]{Vo13}.

\rem{Un certain nombre de nos arguments s'étendent aux fibrations en cubiques de dimension $3$ quelconques. C'est notamment le cas de la partie 2. Deux points seulement ne s'étendent pas et nous empêchent d'obtenir l'analogue des résultats de \cite{Vo13}.

Tout d'abord, on ne connaît pas la conjecture de Tate pour les cycles de codimension $2$ à coefficients dans $\Q_\ell$ pour une fibration en cubiques de dimension $3$. D'autre part, on ne peut pas en général se ramener au cas des classes dont la restriction aux fibres géométriques lisses de la fibration est nulle. Comme dans \cite{Vo13}, cela implique de travailler avec des espaces de modules de courbes sur les cubiques de dimension $3$ pour lesquels la fibre générique de l'application d'Abel-Jacobi est rationnellement connexe. Pour pouvoir appliquer le théorème de Graber, Harris et Starr et obtenir des sections de l'application d'Abel-Jacobi, il est alors nécessaire de passer à une extension finie du corps de base.

Ces deux obstructions étant prises en compte, il est vraisemblablement possible de vérifier que les résultats de \cite{Vo13} montrent l'annulation de la torsion du conoyau de l'application classe de cycle
$$CH^2(X_{\bar k})\otimes \Z_\ell \ra \bigcup_U H^4(X_{\bar k}, \Z_\ell(2))^U,$$
où $U$ parcourt les sous-groupes ouverts du groupe de Galois absolu de $k$, et où $X$ est une fibration en cubiques de dimension $3$ sur une courbe, dont les fibres singulières ont au plus des points doubles ordinaires.
}

\end{document}